\documentclass[12pt]{amsart}

\usepackage{amssymb,amsmath,amsthm,hyperref}
\hypersetup{
pdftitle={Research Paper},
pdfauthor={Abbas Nasrollah Nejad},
pdfsubject={The Aluffi algebra of a hypersurface },
pdfcreator={Abbas Nasrollah Nejad},
colorlinks,
linkcolor=blue,
citecolor=magenta,
anchorcolor=red,
bookmarksopen,
urlcolor=red,
filecolor=red,
pdfpagetransition={Wipe}
}

\theoremstyle{plain}
\newtheorem{Theorem}{Theorem}[section]
\newtheorem{Corollary}[Theorem]{Corollary}

\newtheorem{Proposition}[Theorem]{Proposition}

\newtheorem{Question}[Theorem]{Question}
\theoremstyle{definition}
\newtheorem{Definition}[Theorem]{Definition}
\newtheorem{Example}[Theorem]{Example}
\theoremstyle{Remark}
\newtheorem{Remark}{Remark}

\newcommand{\surjects}{\twoheadrightarrow}

\def\fm{{\mathfrak m}}

\def\NN{\mathbb N}

\def\fp{{\mathfrak p}}
\def\fq{{\mathfrak q}}

\def\hht{{\rm ht}\,}
\def\depth{{\rm depth}\,}

\def\proj#1{{\rm Proj}\, (#1)}

\def\pp{{\mathbb P}}
\def\AA{{\mathbb A}}

\begin{document}

\title[The Aluffi Algebra of a hypersurface with isolated singularity]{The Aluffi Algebra of a hypersurface with isolated singularity}
\author[A. Nasrollah Nejad  ]{  abbas Nasrollah nejad }
%\thanks{This research was in part supported by a grant from IPM (No.900130067).}
\address{department of mathematics institute for advanced studies in basic sciences (IASBS) p.o.box 45195-1159  zanjan, iran.
}
\address{School of Mathematics, Institute for Research in Fundamental Sciences (IPM), P.O.Box 19395-5746, Tehran, Iran
}
\email{abbasnn@iasbs.ac.ir}

\begin{abstract}
The Aluffi algebra is algebraic definition of characteristic cycles of a hypersurface  in intersection theory. In this paper we focus on  the Aluffi algebra of quasi-homogeneous and locally Eulerian hypersurface with isolated singularities. We prove that the Jacobian ideal of an affine hypersurfac with isolated singularities is of linear type if and only if it is locally Eulerian.  We show that the gradient ideal of a projective hypersurface  is of linear type if and only if the corresponding affine curve in the affine chart associated to singular points  is locally Eulerian. We prove that the gradient ideal of the  Nodal and Cuspidal projective plane curves are of linear type. 
\end{abstract}
\maketitle
\section*{Introduction}
 In  \cite{AA} P. Aluffi introduced a graded algebra to describe the characteristic cycle of a hypersurface parallel to well known conormal cycle in intersection theory. Let $Y\subseteq X\subseteq M$, where $X$ is a hypersurface in the smooth variety $M$ and $Y$ is singular subscheme of $X$.  Aluffi constructed a graded $\mathcal{O}_X$-algebra, $\mathrm{qSym}_{X\subset M}(\mathcal{J}_{Y,X})$ which he called \textit{quasi-symmetric algebra}, defined by
\[\mathrm{qSym}_{X\subset M}(\mathcal{J}_{Y,X}):=\mathrm{Sym}_{\mathcal{O}_X}(\mathcal{J}_{Y,X})\otimes_{\mathrm{Sym}_{\mathcal{O}_M}(\mathcal{J}_{Y,M})} \mathcal{R}_{\mathcal{O}_M}(\mathcal{J}_{Y,M}),\]
where $\mathrm{Sym}_{\mathcal{O}_M}(\mathcal{J}_{Y,X})$ is the symmetric algebra of the ideal sheaf $\mathcal{J}_{Y,X}$ of $Y$ in $X$ and  $\mathcal{R}_{\mathcal{O}_M}(\mathcal{J}_{Y,M})$ is the Rees algebra of the ideal sheaf $\mathcal{J}_{Y,M}$ of $Y$ in $M$. Aluffi proved that the characteristic cycle of $X$ is
\begin{equation}\label{CHC}
{\rm Ch}(X)=(-1)^{\dim X}\left[\proj{\mathrm{qSym}_{X\subset M}(\mathcal{J}_{Y,M})}\right]. 
\end{equation}
One of the Aluffi's main result \cite[Theorem 4.4]{AA} is that the (dual) Chern-Mather and the (dual) Schwartz-MacPherson classes of $X$ are (up to sing) the shadows of the cycles $[\proj{\mathcal{R}_{\mathcal{O}_X}(\mathcal{J}_{Y,X})}]$ and $[\proj{\mathrm{qSym}_{X\subset M}(\mathcal{J}_{Y,M})}]$ in $\mathbb{P}(T^*M_{|X})$. 

Inspired by this construction, the author and A. Simis have explored its algebraic flavors, naming it the Aluffi algebra of a pair of ideals \cite{AA}( See also \cite{Thesis1}). Let $R$ be a commutative ring and $J\subset I$ a pair of ideals in $R$. The  \textit{ (embedded) Aluffi algebra} of $I/J$ is 
\[{\mathcal A}_{R/J}(I/J)=\mathrm{Sym}_{R/J}(I/J)\otimes_{\mathrm{Sym}_R(I)} {\mathcal R}_R(I). \]
The Aluffi algebra is algebraic version of characteristic cycle of a hypersurface. As seen in (\cite[Theorem 1.8 and Proposition 2.11]{AA}), in  order that the Aluffi algebra have a certain expected behavior over a regular ambient ring, the ideal $J$ had better be principal. This may be taken as a motivation for this work. Another source for motivation, as mentioned above, is the original work of Aluffi, where he has inquired into the structure of the algebra in the case that the ideal $J=(f)$ is generated by equation of a reduced hypersurface. 

Let  $R=k[x_1,\ldots,x_n]$  be the polynomial ring in $n\geq 2$ variables over an algebraically closed filed $k$ of characteristic zero and  $f\in R$  a reduced polynomial which defines the hypersurface $X=V(f)$. We regard the non-singular hypersurfaces  as uninteresting case and assume that $X$ is singular. The singular subscheme  of the hypersurface $X$ is described  by the ideal $I(f)=(f,J(f))$, the so called \textit{Jacobian ideal} of $f$ where $J(f)=(\partial f/\partial X_1,\ldots, \partial f/\partial X_n)$ is the \textit{gradient ideal} of $f$. 

An ideal $I$ in a commutative ring $A$ is called of \textit{linear type} if the symmetric and the Rees algebra of $I$ are isomorphic. By definition of the Aluffi algebra, if $I(f)$ is of linear type then the Aluffi algebra of $I(f)/(f)$ is isomorphic with the symmetric algebra.  Then by formula (\ref{CHC}) for computing characteristic cycle of the hypersurface $X=V(f)$, we need just compute cycles of the naive blowup. In particular, for Schwartz-MacPherson classes of the hypersurface $X$ we need shadows of the cycles $[\proj{\mathrm{ Sym}_{R/(f)}(I(f)/(f))}]$. By the results of Hunecke and Rossi  \cite{HR} we have more information about the symmetric algebra, especially the  minimal associated primes. The principal question which  motivated this paper is: when the Jacobian ideal of a hypersurface is of linear type? 

Assume that  $f$ is a reduced homogeneous polynomial and the projective hypersurface $X=V(f)\subset \mathbb{P}_k^{n-1}$  has isolated singularities. There is a central structural backstage for the Aluffi algebra and linear type property of $I(f)$ \cite[Corollary 3.2 and Theorem 3.3]{AA}. In fact, the gradient ideal $J(f)$ is of linear type if and only if the coordinates of the vector fields of $\pp_k^{n-1}$ vanishing on $f$ generate an irrelevant ideal in $R$.  

In this work, first we focus on the structure of  the Aluffi algebra of an affine hypersurface, especially the linear type property of the Jacobian and gradient  ideals. The outline of the paper is as follow: In section 1, we study the structure of the Aluffi algebra for quasi-homogeneous hypersurfaces. In Proposition \ref{Quasi-homogeneous}, we prove that the Aluffi algebra of a quasi-homogeneous hypersurface is isomorphic with the symmetric algebra, Aluffi algebra is Cohen-Macaulay and we determine the minimal associated primes. An affine hypersurface $X=V(f)\subseteq \AA_k^n$ with isolated singularity is called \textit{locally Eulerian} if $f\in J(f)_{\fm}$ for every $\fm\in\mathrm{Sing}(X)$. One of the main results is Theorem \ref{LE}. We prove that $X$ is locally Eulerian if and only if $I(f)$ is locally a complete intersection at singular points which equivalent  that $I(f)$ is of linear type. In the sequel, we consider the affine plane curve $X=V(x^a+x^cy^d+y^b)\subseteq \AA_k^2$. We determine when $X$ is locally Eulerian in terms of powers $(a,c,d,b)\in k^4$.  

In section 2, We study linear type property of a reduced projective hypersurface. We say that a reduced projective hypersurface $X=V(f)\subseteq \pp_k^n$ is of \textit{gradient linear type} if $J(f)$ is of linear type. In Corollary \ref{LTcriterion}, we show that $X$ is of gradient linear type if and only if the corresponding affine hypersurface in the affine chart associated to a singular point is locally Eulerian. By using this results in Theorem \ref*{nodal-cuspidal}, we prove that projective plane curves with only  $\mathrm{A}_k$ singularities with $k\geq 1$ are of gradient linear type.  

\medskip

\section{Affine Hypersurfaces}
Throughout this section $k$ is an algebraically closed filed of characteristic zero and $R=k[x_1,\ldots,x_n]$ the polynomial ring over $k$ with $n\geq 2$. Let $f\in R$ be a reduced  polynomial which defines a reduced affine  hypersurface $X=V(f)\subseteq \AA_k^n$.  Let  $$J(f)=(\partial f/\partial x_1,\ldots, \partial f/\partial x_n)\subset R,$$ the so-called the gradient ideal of $f$. The singular locus of $X$ is defined by the Jacobian ideal $I(f)=(f,J(f))$. We say that $X$ has \textit{isolated singularity} if the singular locus of $X$ consist of finitely many points. If $X$ has only isolated singularities  then $\hht I(f)=n$. The \textit{Milnor number} $\mu(f)$ and \textit{Tujurina number} $\tau(f)$ of $f$ at a singular point $\fm$ are the  $k$-vector space dimension of the local Milnor algebra $R_{\fm}/J(f)_{\fm}$ and the local Tjurina algebra  $R_{\fm}/I(f)_{\fm}$, respectively. 
\begin{Remark}\label{Regular_sequence}\rm
If $f\in R$ is a reduced polynomial which defines a hypersurface with isolated singularities. Then by \cite[Lemma 4]{Michler} the codimension of $J(f)_{\fm}$ is $n$. Since $R_{\fm}$ is a local Cohen-Macaulay ring, it follows that  the partial derivatives $\partial f/\partial x_i$  localized at  singular point $\fm$ form an $R_{\fm}$-regular sequence. In particular, $R_{\fm}/J(f)_{\fm}$ is Artinain Gorenstein ring. 
\end{Remark}
In the sequel, we study the structure of the Aluffi  algebra and the Jacobian ideal  for quasi-homogeneous and locally Eulerian hypersurfaces. 
\subsection{Quasi-homogeneous hypersurface}
 Recall that a polynomial $f\in R$ is \textit{quasi-homogeneous} of degree $d$ with weight $r_i$ if it is satisfies :
\[f=\sum_{i=1}^n (r_i/d)x_i\partial f/\partial x_i. \] 

A reduced hypersurface $X\subseteq \AA_k^n$ is called \textit{quasi-homogeneous} if the polynomial defining $X$ is quasi-homogeneous. We have the following central backstage for the Aluffi  algebra of quasi-homogeneous hypersurface with isolated singularities. 

\begin{Proposition}\label{Quasi-homogeneous}
Let $X=V(f)\subseteq \AA_k^n$ be a reduced quasi-homogeneous hypersurface with isolated singularities. Then 
\begin{enumerate}
\item[{\rm (a)}] The Aluffi   algebra of $I(f)/(f)$  is isomorphic with the symmetric algebra. In Particular the Aluffi algebra has the following representation
\[{\mathcal A}_{R/(f)}(I(f)/(f))\simeq R[\mathbf{T}]/(f,\sum_{i=1}^{n}(r_i/d)x_iT_i,\mathcal{L}) \]
where $\mathcal{L}=(\partial f/\partial x_iT_j-\partial f/\partial x_jT_i\ : \  1\leq i<j\leq n)$. 

\item[{\rm (b)}] The Aluffi algebra of $I(f)/(f)$ is Cohen-Macaulay. 

\item[{\rm (c)}] The minimal primes of the Aluffi algebra on ${\mathcal R}_R(I(f))$ are 
\begin{itemize}
\item The minimal prime ideals of ${\mathcal R}_{R/(f)}(I(f)/(f))$ all of the form $\sum_{t\geq 0} (g)\cap I(f)^t$ for a prime factors $g$ of $f$. 
\item The extended ideal $(\mathbf{x}){\mathcal R}_R(I(f))$.

\end{itemize}
\end{enumerate}
\end{Proposition} 
\begin{proof}
(a) Since $f$ is quasi-homogeneous  it follows that $I(f)$ is generated by $R$-regular sequence. Hence $I(f)$ is of linear type and the Aluffi algebra is isomorphic with symmetric algebra. By the definition of the Aluffi algebra one has
\[{\mathcal A}_{R/(f)}(I(f)/(f))\simeq \mathrm{ Sym}_{R/(f)}(I(f)/(f))\simeq \mathrm{ Sym}_R(I(f))/(f,\widetilde{f})\mathrm{ Sym}_R(I(f)),\]
where $\widetilde{f}$ is degree one of $f$ in the symmetric algebra. By using a representation of the symmetric algebra 
\[\mathrm{ Sym}_R(I(f))\simeq R[\mathbf{T}]/I_1([T_1\ \ldots \ T_n].\phi),\] 
where $I_1([T_1\  \ldots\  T_n].\phi)$  is ideal generated by $1$-minors of product of $1\times n$ variable matrix $[T_1\ \ldots \ T_n]$ with the first syzygy matrix $\phi$ of $I(f)$. Since $I(f)$ is generated by a regular sequence then $I(f)$ only have Koszul syzygies. Then all together, we get such representation.

(b) We apply the Criterion \cite[Theorem 10.1]{Trento}. We have to verify the following conditions:
\begin{enumerate}
\item[(I)]  $\mu ((I(f)/(f))_{P/(f)}) \leq \hht (P/(f))+1=\hht(P)$, for every  prime ideal $P$ contain $I(f)$ . 
\item[(II)] ${\rm depth}(H_i)_{P/(f)}\geq \hht (P)- \mu ((I(f)/(f))_{P/(f)})+i-1$, for every  prime ideal $I(f)\subset P$ and every $i$ such that $0\leq i \leq \mu ((I(f)/(f))_{P/(f)})- \hht ((I(f)/(f))_{P/(f)})$ where $H_i$ denotes the $i$th Koszul homology module of the partial derivatives on $R/(f)$. 
\end{enumerate}
The prime ideals contain $I(f)$ are maximal ideal which correspond with singular points. Since $I(f)$ is generated by $n$ elements, then in local ring $R/(f)_{P/(f)}$, we have $\mu ((I(f)/(f))_{P/(f)})=\mu(I(f)_P)=n$ and $\hht ((I(f)/(f))_{P/(f)})=\hht(P/(f))=n-1$.  Then the condition (I), (II) trivially verified. 

(c) This ia an immediate translation of \cite[Propositions 2.8 and 1.9(ii)]{AA}.  
\end{proof}
\begin{Example}\rm
Let $f\in R=k[x,y]$ be a reduced quasi-homogeneous polynomial. By \cite[Proposition 2.5]{ss}, $f$ has the following form:
\[f=ax^{sq}+by^{sp}+\sum_{1\leq r<s}c_rx^{(s-r)q}y^{rp}.\]
where $p,q,s\in \mathbb{N}_{>0}$ such that $\gcd(p,q)=1$. 
The hypersurface   defined by $f$ has singularities if and only if one of the following cases hold:
\begin{itemize}
\item[\rm (1)] $a,b\neq 0$ and $s\geq 2$.
\item[\rm (2)] $a,b=0$ and $s=2$.
\end{itemize}
Then the Aluffi algebra of a quasi-homogeneous plane curve $X=V(f)\subseteq \mathbb{A}_k^2$ is  Cohen-Macaulay and isomorphic with the symmetric algebra.  
\end{Example}

\subsection{Locally Eulerian hypersurface}
We start with the definition of locally Eulerian hypersurface. 
\begin{Definition}
A hypersurface $X=V(f)\subseteq \AA_k^n$ with isolated singularities  is called\textit{ locally Eulerian} if $f\in J(f)_{\fm}$ for every maximal ideal $\fm$ correspond with singular points.  
\end{Definition}
Clearly, any quasi-homogeneous hypersurface is locally Euelrian but the converse is not holds, for example the plane curve defined by $f(x,y)=xy+x^3+y^3\in k[x,y]$ which has one singular point at origin is locally Eulerian, but it is not quasi-homogeneous. If an affine hypersurface $X=V(f)$ is locally Eulerian then the Milnor number of $f$ at each singular point is equal to Tjurina number. 

In what follow we say that a reduced affine  hypersurface $X=V(f)\subset \AA_k^n$  is of \textit{Jacobian  liner type}  if the Jacobian  ideal $I(f)$ is of linear type. The nonsingular and quasi-homogeneous affine hypersurfaces are of  Jacobian linear type. In the following result, we characterize Jacobian linear type affine hypersurfaces with isolated singularities. 
\begin{Theorem}\label{LE}
Let  $X=V(f)\subseteq \AA_k^n$ be a hypersurface with isolated singularities.  The following are equivalent:
\begin{enumerate}
\item[\rm (a)] The hypersurface $X$ is locally Eulerian.
\item[\rm (b)] The Jacobian ideal $I(f)$ is locally a complete  intersection at singular points. 
\item[\rm (c)] The hypersurface $X$  is of Jacobian linear type. 
\end{enumerate}
\end{Theorem}
\begin{proof}
We prove  (a) $\Leftrightarrow$ (b). By  Remark \ref{Regular_sequence}, (a) implies (b). Assume that (b) holds. Let $P\in \AA_k^n$ be a singular point of $X$. By a linear  change of coordinates we may assume that $P$ is origin of $\AA_k^n$. Denote by $\fm$ the  maximal ideal correspond with  origin. The ring $R_{\fm}$ is local so that  the Nakayama's lemma emphasize that the $n$ generators of $I(f)_{\fm}$ may be found among $f,\partial f/\partial x_1,\ldots, \partial f/\partial x_n$.  Assume  that $I(f)_{\fm}=(f,\partial f/\partial x_2,\ldots, \partial f/\partial x_n)_{\fm}$. We have the surjective $k$-algebra homomorphism 
\[R_{\fm}/J(f)_{\fm}\surjects R_{\fm}/I(f)_{\fm} .\] 
Thus the Milnor number $\mu(f)$ at singular point $\fm$ is great or equal than to its Tjurina number $\tau(f)$. Now consider the ideal $(f,x_1)$ which is height 2 complete intersection ideal with an isolated singularity at origin of $\AA_k^n$. Using the L\^e-Greuel formula for the Milnor number of an isolated complete intersection singularity (see \cite{Trung} and \cite{Greuel}), we have
\[\mu(f)+\mu(f,x_1)=\dim_k R_{\fm}/(f,\partial f/\partial x_2,\ldots, \partial f/\partial x_n)_{\fm}. \]
Therefore with above surjection one has
\[\mu(f)\geq \tau(f)=\mu(f)+\mu(f,x_1).\] 
Thus $\mu(f,x_1)=0$ and this implies that $V(f,x_1)\subseteq\AA_k^n$ is non-singular at $\fm$ which is a contradiction. Hence $I(f)_{\fm}=J(f)_{\fm}$ and (a) holds. 

(b)$\Rightarrow$ (c). This is a general fact since linear type property is local and $R/I(f)$ is Artinain Gornestain local ring at each singular point.   Now assume that (b) holds. By \cite[Proposition 2.4]{Trento} for each singular point  $\fm$, one has 
$$\hht I(f)_{\fm}\leq \mu (I(f)_{\fm})\leq \hht \fm= \hht I(f)\leq \hht I(f)_{\fm} ,$$
which shows that $I(f)$ is locally a complete intersection. 
\end{proof}
\begin{Example}\label{notEulerian}\rm
Let $f=x^4-x^2y^2+y^5$. The affine curve $X=V(f)\subseteq\AA_k^2$ has one singular point at origin. 
We can check easily that   $$((\partial f/\partial x, \partial f/\partial y):_{k[x,y]} f)=(5y^2-y,5xy-x,10x^2+y).$$ Hence $f$ can not be locally Eulerian. A computation in \cite{singular} yields that the defining ideal of the Rees algebra of $I(f)$ contains a messy polynomial of  degree $2$ in $T_i$'s. This quadratic polynomial is responsible for the Jacobian ideal failing to be generated by analytically independent elements. The symmetric algebra of $I(f)$ is reduced and Cohen-Macaulay but the Rees algebra is not Cohen-Macaulay.
\end{Example}
\begin{Corollary}
Let $X=V(f)\subseteq \AA_k^n$ be a locally Eulerian affine hypersurface with isolated singularities then
\begin{enumerate}
	\item[\rm{(a)}]  The Aluffi algebra of $I(f)/f$ is isomorphic with the symmetric algebra of $I(f)/f$. In Particular:
	$$\mathcal{A}_{R/(f)}\left(I(f)/(f)\right)\simeq R[T_2,\ldots,T_n]/(f,I_1([T_2\ldots T_n].\phi)),$$
	where  $\phi$ is a matrix obtained by deleting first row of the first syzygy matrix of $I(f)$. 
\item[\rm{(b)}] The Aluffi algebra of $I(f)/f$ is Cohen-Macaulay.
\end{enumerate}
\end{Corollary}
\begin{proof}
(a). The first assertion follows from the Theorem \ref{LE}. Since  $\widetilde{f}=T_1$ then such presentation follows form the proof of  Proposition \ref{Quasi-homogeneous}(a).  For (b), we can use  the same argument as in the Proposition \ref{Quasi-homogeneous}(b).  
\end{proof}
\begin{Example}
Let $X=V(x^a+x^cy^d+y^b)$ be a family of plane curves in $\AA_k^2$ with $a,b,c,d\in \NN_{0}$.  An easy calculation in plane curve theory show that in the following cases $X$ is locally Eulerian singular plane curve:
\begin{enumerate}
\item  $c\geq a$ and $d\geq b$.
\item  $a,b\geq 2$, $c=a-1$ and $d=b-1$.
\item  $a,b\geq 2$, $c=1$ and $d=1$.
\item  $a,b\geq 3$, $c=1$ and $d\leq \frac{b+1}{2}$.
\item  $a,b\geq 3$, $d=1$ and $c\leq \frac{a+1}{2}$.
\item  $a,b\geq 3$, $c=a-1$ and $d\geq \frac{b-1}{2}$.
\item  $a,b\geq 3$, $d=b-1$ and $c\geq \frac{a-1}{2}$.
\item  $a=2$ or $b=2$.
\end{enumerate} 
Now consider the following regions in $k^2$:
\begin{enumerate}
	\item [{\rm (I)}] $\{(c,d)\ |  2\leq c\leq a-2\ ,\ 2\leq d\leq b-2 \}$.
	\item [{\rm (II)}]$\{(c,d)\ | \ c=1, d > \frac{b+1}{2}\} \bigcup \{ (c,d)\ | \ d=b-1 ,  c<\frac{a-1}{2} \}$.
	\item [{\rm (III)}] $\{(c,d)\ | \ d=1, c > \frac{a+1}{2}\} \bigcup \{ (c,d)\ | \ c=a-1 ,  d<\frac{b-1}{2} \}$.
\end{enumerate} 
Assume that $a,b\geq 3$ and  $(c,d)$ belongs to one of the above regions. Then if $X$ is locally Eulerian then it is quasi-homogeneous. Therefore, if $a,b\geq 3$ and $(c,d)$ belongs to regions ${\rm (I),(II),(III)}$ then $X$ is not locally Eulerian.  
\end{Example}

\section{Projective hypersurfaces}
Let $k$ be an algebraically closed filed of characteristic zero and $R=k[x_1,\ldots,x_{n+1}]$ the polynomial ring  with standard grading over $k$. Let $f\in R$ be  a reduced homogeneous polynomial of degree $d\geq 2$. By the Euler formula the singular subscheme of the reduced projective hypersurface $X=V(f)\subseteq \pp^k_n$ is defined by the gradient ideal $ J(f)=(\partial f/\partial x_1,\ldots, \partial f/\partial x_{n+1})$. A reduced projective hypersurface $X=V(f)\subset \pp_k^n$ is called of \textit{gradient  linear type}  if the gradient ideal $J(f)$ is of linear type. It is well know that if $X$ is non-singular then $J(f)$ is irrelevant, i.e.,  generated by regular sequence, and hence is of linear type. Therefore, every non-singular projective hypersurface is of gradient linear type.  

Let $X=V(f)\subseteq \pp_k^n$ be a reduced projective hypersurface defined by $f\in R$ with $\deg(f)=d\geq 2$. Assume that $P\in \pp_k^n$ is a singular point of $X$. By a projective transformation we may assume that $P=[0:0:\ldots: 0:1]$. The ideal $\fp=(x_1,\ldots,x_n)$ is prime ideal correspond with the  point $P$. Consider the affine chart $U_{x_{n+1}}=\AA_k^n$ with coordinate ring $A=k[x_1/x_{n+1},\ldots, x_{n}/x_{n+1}]=k[T_1,\ldots,T_n]$. The equation  of $f$ in this affine chart is $F(T_1,\ldots,T_n):=f(x_1,\ldots,x_{n},1)$. Hence the Jacobian ideal of $F$ is 
$$I(F)=(F,\partial F/\partial T_1,\ldots, \partial F/\partial T_n).$$
Note that the origin of $\AA_k^n$ is the singular point of the affine hypersurface $V(F)\subseteq \AA_k^n$. Denote by $\fq=(T_1,\ldots,T_n)$ the ideal correspond with singular origin.
\begin{Proposition}\label{LC_Affine-projective}
By assumption and  notation as above. The following are equivalent.
\begin{enumerate}
\item[{\rm (a)}] The affine Jacobian ideal $I(F)$ is  a complete intersection locally at $\fq$.
	
\item[{\rm (b)}] The gradient ideal $J(f)$ is  a complete intersection locally at $\fp$.
\end{enumerate}
\end{Proposition}
\begin{proof}
 Let us make some remarks before we address the proof. Setting $f_{x_i}:=\partial f/\partial x_i$ and $F_{T_i}:=\partial F/\partial T_i$.  Denote by $J(f)^{\rm deh}$ dehomogenization of the ideal $J(f)$ with respect to $x_{n+1}$. We have 
 \[J(f)^{\rm deh}=(f_{x_1}^{\rm deh}, \ldots, f_{x_n}^{\rm deh}, f_{x_{n+1}}^{\rm deh} ).   \]
 Clearly, $F_{T_i}=f_{x_i}^{\rm deh}$, for $i=1,\ldots, n$. Using the Euelr formula and dehomogenizing it with respect to $x_{n+1}$ we get
 \begin{equation}\label{Euler_Affine-projective}
 F=\sum_{i=1}^{n}T_i F_{T_i}+f_{x_{n+1}}^{\rm deh}
 \end{equation}
 This relation implies that $I(F)=J(f)^{\rm deh}$. 
Assume that (a) holds. By Theorem \ref{LE}, the polynomial $F$ is locally Eulerian at singular point $\fq$. Hence in the ring $A_{\fq}$ we can write $F=\sum_{i=1}^n G_iF_{T_i}$ where  $G_i\in A_{\fq}\subseteq R_{\fp}$. Using  relation (\ref{Euler_Affine-projective}) we have $f_{x_{n+1}}^{\rm deh}\in (f_{x_1}^{\rm deh}, \ldots, f_{x_n}^{\rm deh})$ locally at $\fq$. Then in $R_{\fp}$ one has 
\[f_{x_{n+1}}=(f_{x_{n+1}}^{\rm deh})^{\rm hom}\in (f_{x_1}^{\rm deh}, \ldots, f_{x_n}^{\rm deh})^{\rm hom}=((f_{x_1},\ldots,f_{x_n}): x_{n+1}^{\infty}).\]
Hence there exists $\alpha\in \mathbb{N}$ such that $x_{n+1}^{\alpha}f_{x_{n+1}}\in (f_{x_1},\ldots,f_{x_n}) $ locally at $\fp$. Since $x_{n+1}$ is unit in $R_{\fp}$ we get $f_{x_{n+1}}\in (f_{x_1},\ldots,f_{x_n}) $. 
Conversely, If  $J(f)_{\fp}\subseteq R_{\fp}$ is a complete intersection, the Nakayama's lemma emphasize that the $n$ generators of $J(f)_{\fp}$ can be find among $f_{x_1},\ldots,f_{x_n},f_{x_{n+1}}$. We claim that $J(f)_{\fp}=(f_{x_1},\ldots, f_{x_{n}} )$. Otherwise, 
Assume that $J(f)_{\fp}=(f_{x_2},\ldots, f_{x_{n+1}} )$. Then in $R_{\fp}$ one has $f_{x_1}=\sum_{i=2}^{n+1}g_if_{x_i}$ with $g_i\in R_{\fp}$. By dehomogenization   with respect to $x_{n+1}$ and using formula (\ref{Euler_Affine-projective}), we get that $F_{T_1}$ locally belongs to $(F,F_{T_2},\ldots, F_{T_n})$ which is impossible by the proof of Theorem \ref{LE}.  Therefore, by formula (\ref{Euler_Affine-projective}) and $I(F)=J(f)^{\rm deh}$ we get the assertion. 
\end{proof} 
\begin{Corollary}\label{LTcriterion}
Let  $X=V(f)\subset \pp_k^n$ be a reduced projective hypersurface with isolated singularities. The following are equivalent:
\begin{itemize}
\item[(1)] The hypersurface $X$ is of gradient linear type.
\item[(2)] Locally at each singular prime the gradient ideal $J(f)$ is a complete intersection.
\item[(3)] For each singular prime $\fp$ such that $x_i\notin \fp$, the Jacobian ideal $I(F)$ in the affine coordinate ring $A=k[x_1/x_i,\ldots,x_{n+1}/x_i]$ is a locally complete intersection at $\fq$ where $F=f(x_1,\ldots,x_{i-1},1,x_{i+1},\ldots,x_{n+1})$ and $\fq$ is maximal ideal associated to $\fp$ in $A$. 
\end{itemize} 
\end{Corollary} 
\begin{proof}
(1) $\Rightarrow$ (2). The same argument as in the proof of Theorem \ref{LE}. 
(2) $\Rightarrow$ (1). The gradient ideal $J(f)$ is almost complete intersection,  $\depth (R/J(f))\geq\dim R/J(f)-1=0$ and by assumption $J(f)_{\fp}$ is a complete intersection at singular primes, then Theorem 5.2 in \cite{AV} complete the proof. 
The equivalence of (2) and (3) follow form Lemma \ref{LC_Affine-projective}. 
\end{proof} 
\begin{Remark}\rm
By \cite[Corollary 3.2]{AA},the conditions of the Theorem (\ref{LTcriterion}) are equivalent to the statement that the coordinates of the vector fields of $\pp^k_n$ vanishing on $f$ generate an irrelevant ideal.  
\end{Remark}
\begin{Example}
Let $X$ be an irreducible cubic surface in $\mathbb{P}_k^3$ with 4 singular point of type node. We show that $X$ is of gradient linear type.  By a projective transformation we may assume that the singular point located at coordinate points:
\[P_1=(1:0:0:0),\ P_2=(0:1:0:0),\ P_3=(0:0:1:0),\ P_4=(0:0:0:1).\]
We should find a polynomial $f$ such that $X=V(f)\subset \pp_k^3$. 

Let  $f\in k[x,y,z,w]$ be  a linear combination of all monomials of degree 3 in $k[x,y,z,w]$ and passes through points $P_1, P_2, P_3, P_4$. Then monomials $x^3, y^3, z^3, w^3$ must be absent from $f$. Since  $P_1$ is singular point of $X$ then all coefficient of linear terms in $y,z,w$ must vanish, i.e., $yx^2,zx^2,wx^2$. Similarly, for singular points $P_2,P_3,P_4$ we have respectively $xy^2,zy^2,wy^2$, $xz^2,yz^2,wz^2$ and $xw^2,yw^2,zw^2$ terms disappearing too. Then the remaining equation is  \[f=a_1xyz+a_2xyw+a_3xzw+a_4yzw.\] 
Note that $a_1.a_2.a_3.a_4\neq 0$, otherwise $f$ reduces, we can assume that $a_1=a_2=a_3=a_4=1$ by scaling $x,y,z,w$. One get $f=xyz+xyw+xzw+yzw$. The surface $X$ is called\textit{ Cayley's nodal cubic surface}.  By symmetry and using the same notation, we show that the affine surface $F=xy+xz+yz+xyz$ is locally Eulerian at singular point $\fp=(x,y,z)$. The gradient ideal of $F$ is generated by 
\[F_x=y+z(1+y),\  F_y=z+x(1+z),\  F_z=x+y(1+x).\]  
We claim that $J(F)_{\fp}=(\fp)_{\fp}$ which implies that $F$ is locally Eulerian. One has
\begin{eqnarray}
\nonumber x(1+(1+x)(1+y)(1+z))&=&F_z-(1+x)F_x+(1+x)(1+y)F_y,\\
\nonumber y(1+(1+x)(1+y)(1+z))&=&F_x-(1+y)F_y+(1+x)(1+y)F_z,\\
\nonumber z(1+(1+x)(1+y)(1+z))&=&F_y-(1+z)F_z+(1+y)(1+z)F_x.
\end{eqnarray}
Since $1+(1+x)(1+y)(1+z)\not\in\fp$ then the claim holds. These calculations also shows that $\mu(f)=\tau(f)=1$ at singular point $\fp$.  
\end{Example}
\begin{Question}\rm
Is any irreducible singular cubic surface in $\pp_k^3$ of gradient linear type?
\end{Question}
\begin{Example}
Let $f=x^4z-x^2y^2z+y^5$. The projective curve $X=V(f)\subseteq\pp_k^2$ has one singular point at $P=[0:0:1]$ with one tangent of multiplicity 2 and two remaining distinct tangents. By Example \ref{notEulerian}, the gradient ideal $J(f)$ is not of linear type.  
\end{Example}
\subsection{Nodal and Cuspidal projective plane curves}
Let $X=V(f)\subseteq \pp_k^2$ be a projective plane curves defined by a reduced polynomial $f\in k[x,y,z]$ of degree $d\geq 2$. If $X$ is singular then singularities of $X$ are isolated. Let $P\in \pp_k^2$ be a singular point of $X$. By a projective transformation we may assume that $P=[0:0:1]$. Write $F(x,y):=f(x,y,1)$, since $P$ is singular point then the multiplicity of $f$ at $P$ ,  ${\rm m}_P(f)\geq 2$. Hence $F(x,y)=F_2+\ldots+F_d$ where $F_i\in k[x,y]$ is a homogeneous polynomial of degree $i$. The simplest case is ${\rm m}_P(f)=2$. For $F_2=ax^2+bxy+cy^2$ there are two cases to be considered. If $b^2\neq4ac$, then $F_2$ has two distinct zeros in $\pp_k^1$ that is, $X$ has two distinct tangents at $P$. In this case, the singular point $P$ is called\textit{ ordinary double point} or \textit{ordinary node}. We can change coordinates so that $F_2=xy$. Then the affine equation of ordinary node is given by 
\[F(x,y)=xy+h(x,y),\]
where ${\rm m}_P(h)\geq 3$. 

For $b^2=4ac$, $F_2$ has a double zero, so $X$ has only one tangent at $P$. In this case, $P$ is called cusp. We can change coordinates so that $F_2=y^2$. Then the affine equation of  cusp is 
\[F(x,y)=y^2+h(x,y)\]
where ${\rm m}_P(h)\geq 3$. Note that in two  above cases, the intersection multiplicity ${\rm mult}_P(X,L)\geq 3$ where $L$ is the tangent line at $P$. If ${\rm mult}_P(X,L)=3$ then $P$ is called simple ordinary node and simple cusp.

A projective plane curve $X\subset\pp_k^2$ is called \textit{Nodal} if  
its singular points are node and is called \textit{Cuspidal} if its singular points are cusp. 
\begin{Theorem}\label{nodal-cuspidal}
Nodal and Cuspidal curves  are  of gradient linear type. 
\end{Theorem}
\begin{proof}
Let $X\subset \pp_k^n$ stands for a projective  curve defined by reduced  homogeneous polynomial $f\in R=k[x,y,z]$ of degree $d\geq 3$.  Let $P\in \pp_k^2$ be a singular point of $X$. Denote by $\fq\in R$ the prime ideal of $P$. By projective transformation we may assume that $P=[0:0:1]$ and then $\fq=(x,y)$. We apply Corollary \ref{LTcriterion}, so that we prove  the affine equation of nodal and cuspidal curves  are  locally Eulerian at singular point $\fq=(x,y)$. 

The affine equation of nodal curve is $F(x,y)=xy+h(x,y)$
, where $h(x,y)$ has initial degree at least 3. The Jacobian ideal of $F$ is 
$$J(F)=(x+\partial h/\partial x,\ y+\partial h/\partial y).$$
We can write 
\[\partial h/\partial x=yG_1(x,y)+G_2(x)\quad ,\quad \partial h/\partial y=xL_1(x,y)+L_2(y), \]
where $G_1(0,0)=L_1(0,0)=0$ and $G_2\in k[x], L_2\in k[y]$ have initial degree at least 2. Thus
\[J(F)=(y(1+G_1(x,y))+G_2(x), x(1+L(x,y))+L_2(y)).\] 
Note that $1+G_1$ and $1+L_1$ are unit locally at $\fq$. Thus the gradient ideal locally at $\fq$ is generated by polynomials $x-\alpha L_2(y), y-\beta G_2(x) $, with $\alpha,\beta $ unit and $G, L$ of initial degree at least $2$. We claim that in $k[x,y]_{\fq}$:
 $$(x-\alpha L(y), y-\beta G(x))_{\fq}=(x,y)_{\fq}.$$ 
We can write $L(y)=y\ell(y)$ and $G(x)=x\wp(x)$. We have
$$x(1-\alpha\beta\ell(y)\wp(x))=(x-\alpha y\ell(y))+\alpha\ell(y)(y-\beta x\wp(x))$$
and $$y(1-\alpha\beta\ell(y)\wp(x))=(y-\beta x\wp(x))+\beta \wp(x)(x-\alpha y\ell(y)).$$
Since the element  $1-\alpha\beta\ell(y)\wp(x)$ is unit locally at $\fq$ this proves the claim. Therefore, $F$ is locally Eulerian. 

For a cuspidal curve the affine equation is given by $F(x,y)=y^2+h(x,y)$, where $h(x,y)$  has initial degree at least 3. The tangent line at point $P$ is $L=V(y)$ and we assume that ${\rm mult}_P(X\cap L)=3$. This condition means that $y$ is not divisor of 
$$F_3=d_1x^3+d_2x^2y+d_3xy^2+d_4y^33.$$
Then $d_1\neq 0$ and we may assume that $d_1=-1$. This gives 
\[F(x,y)=y^2-x^3+yg(x,y)+h(x,y),\]
where $g(x,y)$ is homogeneous of degree $2$ and the initial degree of $h(x,y)$ is at least $4$, as the affine equation of a simple cusp. 
 The Jacobian ideal of $F$ is generated by 
 \begin{eqnarray}
\nonumber \partial F/\partial x&=& -3x^2+y\partial g/\partial x+\partial h/\partial x,\\
\nonumber  \partial F/\partial y&=& 2y+g(x,y)+ y\partial g/\partial y+\partial h/\partial y.
 \end{eqnarray}
We can write 
\[\partial h/\partial x=x^2G_1(x,y)+xG_2(y)+G_3(y)\quad, \quad \partial h/\partial y= yH_1(x,y)+H_2(x), \]
where $G_1(0,0)=H_1(0,0)=0$, the initial degree  of $G_2(y)$ is at least $2$ and the initial degree of $G_3(y)$ and $H_2(x)$ is at least $3$. Using these relation and renaming $G_i$ and $H_i$, one has 
\begin{eqnarray}
\nonumber  \partial F/\partial x&=& -3x^2+ x^2G_1(x,y)+xG_2(y)+G_3(y),\\
\nonumber  \partial F/\partial y&=& 2y+yH_1(x,y)+H_2(x),
 \end{eqnarray}
where $G_1(0,0)=H_1(0,0)=0$, the initial degree of $G_2(y)$ is at least $1$, the initial degrees of $G_3(y)$ and $H_2(x)$ is at least $2$. 

We have 
\[ J(F)=\left(x^2(-3+G_1(x,y))+xG_2(y)+G_3(y),\ y(2+H_1(x,y))+H_2(x)\right).\]
Note that the elements $-3+G_1(x,y),\ 2+H_1(x,y)$ are unit locally at $\fq$. Thus the gradient ideal ideal $J(F)$ locally at $\fq$ is generated by $x^2-\alpha(xG_2(y)+G_3(y)), \ y-\beta H_2(x).$
We claim that 
\[(x^2-\alpha(xG_2(y)+G_3(y)), \ y-\beta H_2(x))_{\fq}=(x^2,y)_{\fq}. \]
The condition on initial degree of $G_2,G_3,H_2$ implies the one side inclusion. For the converse, by condition on initial degree, we write 
\[xG_2(y)+G_3(y)=y(xL_2(y)+L_3(y))\quad,\quad H_2(x)=x^2Q(x). \]
We have 
\begin{eqnarray}
\nonumber Ux^2 &=& x^2-\alpha(xG_2(y)+G_3(y))\  +\ \alpha (xL_2(y)+L_3(y))(y-\beta H_2(x))\\
\nonumber Uy &= &y-\beta H_2(x)\ +\ \beta Q(x)(x^2-\alpha(xG_2(y)+G_3(y))  
\end{eqnarray}
 Since the element $U=(1-\alpha\beta Q(x)(xL_2(y)+L_3(y)))$ is unit locally ay $\fq$, this complete the assertion. 
 
If ${\rm mult}_P(X\cap L)=k>3$ then $y$ is not a divisor of $F_k$, homogeneous polynomial of degree $k$, then the coefficient of monomial $x^k$ is nonzero and we may assume that it is $-1$. This gives 
\[F(x,y)=y^2-x^k+yg(x,y)+h(x,y)\]
where $g(x,y)$ is homogeneous of degree $k-1$ and the initial degree of $h(x,y)$ is at least $k+1$. The same arguments as $k=3$ shows that in this case the gradient ideal locally at singular point $\fq=(x,y)$ is generated by $(x^{k-1},y)$ which is a complete intersection in $k[x,y]_{\fq}$.  
\end{proof}
\begin{Corollary}\label{conic-cubic}
Any projective cubic curve in $\pp_k^2$ is of gradient linear type. 
\end{Corollary}
\begin{proof}
If $X$ is irreducible singular cubic then $X$ has at most one singular point of multiplicity 2 which is either a node or a cusp. Thus by Theorem \ref{nodal-cuspidal}, $X$ is of gradient linear type.

If $X$ is a reducible  cubic then it is of the form an irreducible conic and a line not tangent to conic  or of the form an irreducible conic and a line tangent to conic. In two cases, the reducible cubic has singular points of type node and cusp, respectively, hence is of gradient linear type. 

If $X$ is a cubic which reduced to three non-concurrent line then $X$ has $3$ singular points of type node. Assume that  $X=V(f)\subseteq \pp_k^2$ is a cubic which  reduced to three concurrent line. Thus it is projectively equivalent to $f=(y-z)(z-x)(x-y)$. An easy calculation show that the affine curve  is locally  Eulerian at singular point. 
\end{proof}

\begin{Remark}\label{germs-analytical}\rm
\begin{enumerate}
\item 
The local plane curve calculation in Nodal case is as same as one which A. Simis in \cite{ASimis} used to prove that $\depth(R/J(f))=0$. 
\item
Suppose that $X=V(f)\subseteq \pp_k^2$ is projective plane curve of multiplicity ${\rm m}_p(f)=2$ at the point $P=[0:0:1]$. The affine curve $X_z:=V(F(x,y))$ is an analytic subset of $\AA_k^2$ and $(X_z,(0,0))$ is an analytic set germ. If this set germ is analytically equivalent to $(V(x^{k+1}-y^{2}),(0,0))$ for some $k\in \mathbb{N}$, then $P$ is called an ${\rm A}_k$ singularity of $X$. By [Theorems 2.46 and 2.48]\cite{MCE} and the proof of Theorem \ref{nodal-cuspidal}, nodes singularities are ${\rm A}_1$, ordinary cusps are ${\rm A}_2$ and cusps with condition ${\rm mult}_P(X\cap L)=k$ are $A_k$ for $k\geq 2$, where $L$ is tangent line at $P$. The Theorem  \ref{nodal-cuspidal} shows that any projective plane curve only with  ${\rm A}_k$ singularity is of gradient linear type.
\item If $X=V(f)\pp_k^2$ is Nodal or Cuspidal projective curve then the Aluffi algebra of $J(f)/(f)$ is isomorphic with the symmetric algebra and it is Cohen-Macaulay \cite[Theorem 3.3 ]{AA}. 
\end{enumerate}
\end{Remark}
\section*{Acknowledgement} I would like to thanks  Aron Simis for reading the manuscript and helpful comments on the subject to make the paper more comprehensive. 

\end{document}